\documentclass[11pt]{amsart}
\usepackage[a4paper,left=1.15in,right=1.15in,top=1.15in,bottom=1.15in]{geometry}
\usepackage{amsmath,amssymb,amsthm,mathtools}
\usepackage{enumitem}
\usepackage{graphicx}
\usepackage{subcaption}
\usepackage{placeins}
\usepackage{microtype}
\usepackage{hyperref}
\hypersetup{
  colorlinks=true,
  linkcolor=blue,
  citecolor=blue,
  urlcolor=blue
}
\usepackage{tikz}
\usetikzlibrary{arrows.meta}

\newtheorem{theorem}{Theorem}[section]
\newtheorem{corollary}[theorem]{Corollary}
\newtheorem{proposition}[theorem]{Proposition}
\newtheorem{lemma}[theorem]{Lemma}

\theoremstyle{definition}
\newtheorem{definition}[theorem]{Definition}

\newtheoremstyle{boldremark}
  {3pt}
  {3pt}
  {\normalfont}
  {}
  {\bfseries}
  {.}
  { }
  {}

\theoremstyle{boldremark}
\newtheorem{remark}[theorem]{Remark}

\theoremstyle{plain}

\newcommand{\Th}{\operatorname{Th}}
\newcommand{\Fix}{\operatorname{Fix}}
\newcommand{\Aut}{\operatorname{Aut}}
\newcommand{\Bur}{\overline{\mathcal B}}
\newcommand{\Dhalf}{\mathfrak D}
\newcommand{\ZZ}{\mathbb Z}
\newcommand{\QQ}{\mathbb Q}

\newcommand{\doteqeq}{\doteq}
\newcommand{\Res}{\operatorname{Res}}
\newcommand{\sqfree}{\operatorname{sf}}

\title[Equivariant rational sliceness of Turk's head knots]
{Equivariant rational sliceness and Klein amphichirality of odd-stranded Turk's head knots}

\author{Suman Saurabh}
\email{realsumansaurabh@gmail.com}

\subjclass[2020]{57K10, 57K14}
\keywords{Turk's head knots, equivariant rational sliceness, amphichirality,
strong inversion, Gauss diagrams, abstract link diagrams, disk-band surfaces,
Burau representation, Alexander polynomial}

\begin{document}

\begin{abstract}
We establish a sharp parity dichotomy for the equivariant \(\mathbb Q\)-sliceness
and Klein amphichirality of odd-stranded Turk's head knots.  When the twisting
parameter \(q\) is odd, we construct a commuting pair of ambient involutions by
lifting explicit symmetries of the signed, arrowed Gauss diagram through the
associated abstract link diagram and supporting sphere.  This proves that the
knots are Klein amphichiral and hence equivariantly \(\mathbb Q\)-slice.  When
\(q\) is even, we prove that the knots are not equivariantly \(\mathbb Q\)-slice
with respect to any strong inversion.  The obstruction is the equivariant
Fox--Milnor square condition of Di Prisa--\c{S}avk: we show that the Alexander
polynomial is not a Laurent square. The even-\(q\) argument combines a Burau square factorization arising from
Garside conjugacy with a mod-\(2\) Seifert-matrix computation.
\end{abstract}

\maketitle

\section{Introduction}\label{sec:introduction}

Equivariant rational sliceness is a refinement of rational sliceness for
strongly invertible knots. For a strongly invertible knot \((K,\rho)\), the
pair is equivariantly \(\mathbb Q\)-slice if \(\rho\) extends over a rational
homology \(4\)-ball in which \(K\) bounds a properly embedded invariant disk.
We shall use the following precise formulation.

\begin{definition}\label{def:eq-Q-slice}
Let \((K,\rho)\) be a strongly invertible knot. We say that
\((K,\rho)\) is \emph{equivariantly \(\mathbb Q\)-slice} if there exist
a compact connected oriented smooth \(4\)-manifold \(W\), a smoothly and
properly embedded disk \(\Delta\subset W\), and an orientation-preserving
smooth involution
\[
        \widetilde{\rho}:W\longrightarrow W
\]
such that
\[
        \partial W=S^3,
        \qquad
        H_*(W;\mathbb Q)\cong H_*(B^4;\mathbb Q),
\]
and
\[
        \partial\Delta=K,
        \qquad
        \widetilde{\rho}|_{\partial W}=\rho,
        \qquad
        \widetilde{\rho}(\Delta)=\Delta.
\]
The manifold \(W\) is then called a \(\mathbb Q\)-homology \(4\)-ball.
\end{definition}

For \(p\geq2\), set
\[
        A_p=
        \prod_{i=1}^{p-1}\sigma_i^{(-1)^{i-1}}
        \in B_p,
\]
and define
\[
        \Th(p,q)=\widehat{A_p^q}.
\]
When \(\gcd(p,q)=1\), this is a knot. An oriented knot \(K\subset S^3\) is called \emph{strongly invertible}
if there exists an orientation-preserving involution
\(\rho:S^3\to S^3\) such that
\[
        \rho(K)=-K.
\]
In this situation,
\[
        \Fix(\rho)\cong S^1,
        \qquad
        \Fix(\rho)\cap K\cong S^0.
\]

An oriented knot \(K\subset S^3\) is called
\emph{strongly negative amphichiral} if there exists an
orientation-reversing involution
\(\tau:S^3\to S^3\) such that
\[
        \tau(K)=-K.
\]
The fixed set of such an involution is homeomorphic to either \(S^0\)
or \(S^2\). In the construction below, \(\tau\) is of the \(S^0\)-fixed-set
type and satisfies
\[
        K\cap\Fix(\tau)=\Fix(\tau).
\]
In this case, we shall also say that \(\tau\) is a strongly negative
amphichiral involution for \(K\). 

Equivariant sliceness of strongly invertible knots goes back to
Sakuma~\cite{Sakuma1986}, while the rational sliceness of strongly negative
amphichiral knots was established by Kawauchi~\cite{Kawauchi2009}; Levine later
showed that the resulting Kawauchi manifold is independent of the knot up to
diffeomorphism~\cite{Levine2023}.

Following Di Prisa and \c{S}avk, a knot is
Klein amphichiral if it admits a strong inversion and a strongly negative
amphichiral involution which commute.

Di Prisa and \c{S}avk studied Klein amphichirality in connection with
equivariant \(\mathbb Q\)-sliceness of strongly invertible knots. They proved
that every Klein amphichiral knot is equivariantly \(\mathbb Q\)-slice in a
single \(\mathbb Q\)-homology \(4\)-ball~\cite{DPSequivariant}. They also proved
an equivariant Fox--Milnor obstruction: if a strongly invertible knot is
equivariantly \(\mathbb Q\)-slice, then its Alexander polynomial is a Laurent
square~\cite{DPSequivariant}. Their survey on Turk's head knots and links
records that \(\Th(p,q)\) is always strongly invertible and is strongly negative
amphichiral when \(p\) is odd~\cite{DPSsurvey}. Thus, for odd strand number,
the remaining question is not the separate existence of these symmetries, but
whether they can be made compatible in the sense needed for Klein
amphichirality and equivariant rational sliceness.

The answer depends sharply on the parity of \(q\). When \(q\) is odd, we show
that the strong inversion and the strongly negative amphichiral symmetry can be
chosen to commute, giving Klein amphichirality. When \(q\) is even, the same
conclusion is impossible: the equivariant Fox--Milnor obstruction of Di
Prisa--\c{S}avk forces the Alexander polynomial of any equivariantly
\(\mathbb Q\)-slice strongly invertible knot to be a Laurent square, and we
prove that \(\Delta_{\Th(p,q)}(t)\) is not a Laurent square for even \(q\).
Thus the parity of \(q\) separates the constructive and obstructive sides of
the problem.

The main result of this paper gives a complete parity dichotomy for odd-stranded
Turk's head knots. In the odd-\(q\) case we construct an explicit
Klein-amphichiral \(V_4\)-symmetry, while in the even-\(q\) case we obstruct
equivariant \(\mathbb Q\)-sliceness by proving that the full Alexander
polynomial is not a Laurent square. For the three-stranded family
\(\Th(3,q)\), the corresponding dichotomy was previously established by
Di Prisa--\c{S}avk~\cite{DPSequivariant} using different methods; the present
paper extends this dichotomy to all odd strand numbers.

\begin{theorem}[Parity dichotomy]\label{thm:parity-dichotomy}
Let \(p\geq3\) be odd, let \(q\geq2\), and suppose \(\gcd(p,q)=1\). Then:
\begin{enumerate}[label=\textup{(\roman*)}]
    \item if \(q\) is odd, then \(\Th(p,q)\) is Klein amphichiral. In particular,
    with respect to the strong inversion constructed below, \(\Th(p,q)\) is
    equivariantly \(\mathbb Q\)-slice;
    \item if \(q\) is even, then \(\Th(p,q)\) is not equivariantly
    \(\mathbb Q\)-slice with respect to any strong inversion. In particular,
    \(\Th(p,q)\) is not Klein amphichiral.
\end{enumerate}
\end{theorem}

The two halves of the dichotomy arise from rather different structures.
For odd \(q\), the result is controlled by a compatible \(V_4\)-action on
the Gauss data and its supporting-surface realization. For even \(q\), the
obstruction is independent of the chosen strong inversion and is detected by
a square class of the Alexander polynomial. Thus the parity distinction is
not merely diagrammatic: it separates the existence of a commuting geometric
symmetry from a global concordance obstruction.

The constructive half of Theorem~\ref{thm:parity-dichotomy} is proved by an
explicit construction. More precisely, we prove the following theorem.

\begin{theorem}\label{thm:odd-klein}
Let \(p,q\geq3\) be odd integers with \(\gcd(p,q)=1\). Then
\[
        K_{p,q}=\Th(p,q)=\widehat{A_p^q}
\]
is Klein amphichiral. More precisely, there exist smooth ambient involutions
\[
        \rho,\tau:S^3\longrightarrow S^3
\]
which generate a subgroup
\[
        \langle\rho,\tau\rangle\cong \mathbb Z_2\oplus\mathbb Z_2
\]
and satisfy the following properties:
\begin{enumerate}[label=\textup{(\roman*)}]
    \item \(\rho\) is orientation-preserving and satisfies
    \[
        \rho(K_{p,q})=-K_{p,q},
        \qquad
        \Fix(\rho)\cong S^1,
        \qquad
        \#\bigl(K_{p,q}\cap\Fix(\rho)\bigr)=2;
    \]

    \item \(\tau\) is orientation-reversing and satisfies
    \[
        \tau(K_{p,q})=-K_{p,q},
        \qquad
        \Fix(\tau)\cong S^0,
        \qquad
        K_{p,q}\cap\Fix(\tau)=\Fix(\tau);
    \]

    \item \(\rho\tau=\tau\rho\).
\end{enumerate}
\end{theorem}

\begin{corollary}\label{cor:equivariant-Q-slice}
Let \(p,q\geq3\) be odd and suppose \(\gcd(p,q)=1\). With respect to the strong
inversion \(\rho\) constructed in Theorem~\ref{thm:odd-klein}, the strongly
invertible knot \((\Th(p,q),\rho)\) is equivariantly \(\mathbb Q\)-slice.
\end{corollary}

\begin{proof}
By Theorem~\ref{thm:odd-klein}, \(\Th(p,q)\) is Klein amphichiral. Di
Prisa--\c{S}avk proved that every Klein amphichiral knot is equivariantly
\(\mathbb Q\)-slice~\cite{DPSequivariant}. The conclusion follows.
\end{proof}

The obstructive half is a full Alexander-polynomial obstruction, not merely a
determinant obstruction.

\begin{theorem}\label{thm:even-obstruction}
Let \(p\geq3\) be odd and let \(q=2m\) be even. Suppose \(\gcd(p,q)=1\). Then
\[
        \Delta_{\Th(p,q)}(t)
\]
is not a Laurent square up to multiplication by a unit in
\(\ZZ[t,t^{-1}]\). Consequently, \(\Th(p,q)\) is not equivariantly \(\QQ\)-slice with
respect to any strong inversion.
\end{theorem}

\begin{corollary}\label{cor:even-not-klein}
Let \(p\geq3\) be odd and let \(q\) be even, with \(\gcd(p,q)=1\). Then
\(\Th(p,q)\) is not Klein amphichiral.
\end{corollary}

The standard braid word is useful for computing the Gauss code, but it is too
rigid as a spatial model for the full Klein symmetry. In a standard rectangular braid diagram, the local braid generators are supported on the
adjacent edges of the path graph
\[
        P_p:\quad 1-2-\cdots-p.
\]
Any symmetry of this rectangular braid model which sends standard braid
generators to standard braid generators induces an automorphism of \(P_p\). Since
\[
        \Aut(P_p)\cong \mathbb Z_2,
\]
the only nontrivial such operation is full reversal of the strand order. Thus
one should not expect the required commuting pair to be visible as a literal
symmetry of the standard rectangular braid diagram. This is why the construction of
Theorem~\ref{thm:odd-klein} is carried out on the signed, arrowed Gauss diagram
and then realized through the associated abstract link diagram. We use the
language of abstract link diagrams of Kamada--Kamada~\cite{KamadaKamada2000}.
In this framework, the Gauss data together with the cyclic order data determine
a disk-band surface, and capping the boundary components gives a closed
supporting surface. The related viewpoint of knots on surfaces and stable
equivalence is due to Carter--Kamada--Saito~\cite{CKS2002}. In the present
case, the Gauss diagram comes from the standard classical closed-braid diagram,
so the closed supporting surface is \(S^2\).

The odd-\(q\) construction then proceeds in three steps. First, we define two
commuting involutions \(R\) and \(T\) on the signed, arrowed Gauss diagram.
Second, we verify that they are compatible with the cyclic order data of the
associated abstract link diagram, so that they act on the supporting sphere.
Third, we thicken the supporting sphere to \(S^3\) and use normal reversal to
convert the exchange of over- and under-passing branches into ambient
involutions preserving the embedded knot.

The construction is carried out first in the piecewise-linear category.
Lemma~\ref{lem:equivariant-smoothing} then equips the resulting finite
\(V_4\)-action with a compatible smooth structure and smooths the invariant
knot equivariantly. Consequently, the final ambient involutions in
Theorem~\ref{thm:odd-klein} may be taken to be smooth.

The paper is organized as follows. Sections~\ref{sec:Gauss_code}--\ref{sec:fixed_sets}
prove Theorem~\ref{thm:odd-klein}: we compute the Gauss code, verify compatibility
with the cyclic order data of the associated abstract link diagram, realize the
induced \(V_4\)-action on the supporting sphere obtained by capping the
disk-band surface, thicken it to \(S^3\), and identify the fixed sets.
Sections~\ref{sec:laurent-prelim}--\ref{sec:even-failure} prove
Theorem~\ref{thm:even-obstruction}: first by establishing a Burau square
quotient, and then by proving the base case \(q=2\) is not a Laurent square
modulo \(2\). Section~\ref{sec:examples} records determinant checks, and
Section~\ref{sec:conclusion} summarizes the dichotomy.

\section{The Gauss code}\label{sec:Gauss_code}

Let
\[
        A_p=
        \sigma_1\sigma_2^{-1}\sigma_3\sigma_4^{-1}
        \cdots
        \sigma_{p-2}\sigma_{p-1}^{-1}.
\]
Since \(p\) is odd, \(A_p\) has \(p-1\) letters, an even number, and their signs
alternate.

Let
\[
        r=0,1,\ldots,p-2
\]
index the letters of \(A_p\), so that the \(r\)-th letter is
\[
        \sigma_{r+1}^{(-1)^r}.
\]
For \(j\in\mathbb Z/q\mathbb Z\), let \(c_{j,r}\) denote the crossing
corresponding to the \(r\)-th letter in the \(j\)-th copy of \(A_p\). Its sign
is
\[
        \varepsilon(c_{j,r})=(-1)^r.
\]

For each \(c_{j,r}\), let
\[
        u_{j,r},\qquad v_{j,r}
\]
be its two preimages on the oriented Gauss circle, where \(u_{j,r}\) is the
branch entering from the lower-index braid position and \(v_{j,r}\) is the
branch entering from the higher-index braid position. 

Each chord is oriented from its over-passing endpoint to its
under-passing endpoint. Together with the crossing sign, this gives
the signed, arrowed Gauss diagram \(\mathcal G_{p,q}\). Set
\[
        E_{p,q}=
        \{u_{j,r},v_{j,r}\mid j\in\mathbb Z/q\mathbb Z,\ 0\leq r\leq p-2\}.
\]
Thus \(|E_{p,q}|=2(p-1)q\).

Let
\[
        S:E_{p,q}\longrightarrow E_{p,q}
\]
be the successor map along the oriented knot.

\begin{lemma}\label{lem:successor}
The successor map is given by
\[
        S(u_{j,r})=u_{j,r+1}
        \qquad (0\leq r\leq p-3),
\]
\[
        S(u_{j,p-2})=v_{j+1,p-2},
\]
\[
        S(v_{j,0})=u_{j+1,0},
\]
and
\[
        S(v_{j,r})=v_{j+1,r-1}
        \qquad (1\leq r\leq p-2),
\]
where all \(j\)-indices are taken modulo \(q\).
\end{lemma}

\begin{proof}
The \(r\)-th letter of \(A_p\) involves the adjacent braid positions
\(r+1\) and \(r+2\). If a strand enters \(c_{j,r}\) from the lower-index
position, then after the crossing it moves to the higher-index position. For
\(0\leq r\leq p-3\), the next crossing it meets is \(c_{j,r+1}\), again from
the lower-index side. Hence
\[
        S(u_{j,r})=u_{j,r+1}.
\]
For \(r=p-2\), the strand exits the block in position \(p\). In the next copy
of \(A_p\), the next crossing it meets is again the last crossing, from the
higher-index side. Therefore
\[
        S(u_{j,p-2})=v_{j+1,p-2}.
\]

If a strand enters \(c_{j,0}\) from the higher-index position, then after the
crossing it moves to position \(1\). The next crossing it meets is the first
crossing in the next block, from the lower-index side. Hence
\[
        S(v_{j,0})=u_{j+1,0}.
\]
For \(1\leq r\leq p-2\), entering \(c_{j,r}\) from the higher-index side sends
the strand to position \(r+1\). The next crossing it meets occurs in the next
block at generator \(r-1\), from the higher-index side. Thus
\[
        S(v_{j,r})=v_{j+1,r-1}.
\]
\end{proof}

\begin{remark}\label{rem:one-component}
Passing through one block \(A_p\) sends a braid position \(k\) to \(k-1\)
modulo \(p\). Passing through \(q\) blocks sends \(k\) to \(k-q\) modulo \(p\).
Since \(\gcd(p,q)=1\), the closure is a knot; equivalently, \(S\) has one
orbit on \(E_{p,q}\), of length \(2(p-1)q\).
\end{remark}

Since \(q\) is odd, \(2\) is invertible modulo \(q\). Define
\[
        \lambda\equiv (p-2)2^{-1}\pmod q,
\]
so that
\[
        2\lambda\equiv p-2\pmod q.
\]

Define maps
\[
        T,R:E_{p,q}\longrightarrow E_{p,q}
\]
by
\[
        T(u_{j,r})=u_{-j,p-2-r},\qquad
        T(v_{j,r})=v_{-j,p-2-r},
\]
and
\[
        R(u_{j,r})=v_{\lambda-j-r,r},\qquad
        R(v_{j,r})=u_{\lambda-j-r,r}.
\]

\begin{lemma}\label{lem:reverse-circle}
The maps \(T\) and \(R\) reverse the orientation of the Gauss circle:
\[
        TS=S^{-1}T,\qquad RS=S^{-1}R.
\]
\end{lemma}

\begin{proof}
For \(T\), the four cases are as follows. If \(0\leq r\leq p-3\), then
\[
        TS(u_{j,r})
        =
        T(u_{j,r+1})
        =
        u_{-j,p-3-r}
        =
        S^{-1}(u_{-j,p-2-r})
        =
        S^{-1}T(u_{j,r}).
\]
For \(r=p-2\),
\[
        TS(u_{j,p-2})
        =
        T(v_{j+1,p-2})
        =
        v_{-j-1,0}
        =
        S^{-1}(u_{-j,0})
        =
        S^{-1}T(u_{j,p-2}).
\]
For \(r=0\),
\[
        TS(v_{j,0})
        =
        T(u_{j+1,0})
        =
        u_{-j-1,p-2}
        =
        S^{-1}(v_{-j,p-2})
        =
        S^{-1}T(v_{j,0}).
\]
If \(1\leq r\leq p-2\), then
\[
        TS(v_{j,r})
        =
        T(v_{j+1,r-1})
        =
        v_{-j-1,p-1-r}
        =
        S^{-1}(v_{-j,p-2-r})
        =
        S^{-1}T(v_{j,r}).
\]
Thus \(TS=S^{-1}T\).

For \(R\), if \(0\leq r\leq p-3\), then
\[
        RS(u_{j,r})
        =
        R(u_{j,r+1})
        =
        v_{\lambda-j-r-1,r+1}
        =
        S^{-1}(v_{\lambda-j-r,r})
        =
        S^{-1}R(u_{j,r}).
\]
For \(r=p-2\),
\[
        RS(u_{j,p-2})
        =
        R(v_{j+1,p-2})
        =
        u_{\lambda-j-(p-1),p-2}
        =
        S^{-1}(v_{\lambda-j-(p-2),p-2})
        =
        S^{-1}R(u_{j,p-2}).
\]
For \(r=0\),
\[
        RS(v_{j,0})
        =
        R(u_{j+1,0})
        =
        v_{\lambda-j-1,0}
        =
        S^{-1}(u_{\lambda-j,0})
        =
        S^{-1}R(v_{j,0}).
\]
If \(1\leq r\leq p-2\), then
\[
        RS(v_{j,r})
        =
        R(v_{j+1,r-1})
        =
        u_{\lambda-j-r,r-1}
        =
        S^{-1}(u_{\lambda-j-r,r})
        =
        S^{-1}R(v_{j,r}).
\]
Hence \(RS=S^{-1}R\).
\end{proof}

\begin{lemma}\label{lem:commuting}
The maps \(R\) and \(T\) are commuting involutions.
\end{lemma}

\begin{proof}
The identities \(T^2=1\) and \(R^2=1\) follow immediately:
\[
        T^2(u_{j,r})=u_{j,r},\qquad T^2(v_{j,r})=v_{j,r},
\]
and
\[
        R^2(u_{j,r})
        =
        R(v_{\lambda-j-r,r})
        =
        u_{\lambda-(\lambda-j-r)-r,r}
        =
        u_{j,r},
\]
with the same calculation for \(v_{j,r}\).

It remains to prove commutativity. On \(u_{j,r}\),
\[
        RT(u_{j,r})
        =
        R(u_{-j,p-2-r})
        =
        v_{\lambda+j-(p-2-r),p-2-r}
        =
        v_{\lambda+j+r-(p-2),p-2-r},
\]
whereas
\[
        TR(u_{j,r})
        =
        T(v_{\lambda-j-r,r})
        =
        v_{-(\lambda-j-r),p-2-r}
        =
        v_{j+r-\lambda,p-2-r}.
\]
The two \(j\)-coordinates agree modulo \(q\) because
\[
        2\lambda\equiv p-2\pmod q.
\]
Thus \(RT(u_{j,r})=TR(u_{j,r})\). The calculation on \(v_{j,r}\) is identical.
Therefore \(RT=TR\).
\end{proof}

By Lemma~\ref{lem:reverse-circle}, each of \(R\) and \(T\) sends every
complementary arc of the marked Gauss circle to a complementary arc, reversing
its orientation. We extend \(R\) and \(T\) linearly over these complementary
arcs and continue to denote the resulting PL involutions of the Gauss circle
by the same letters. By Lemma~\ref{lem:commuting}, these extensions commute.

\begin{lemma}\label{lem:crossings}
The maps \(T\) and \(R\) preserve chord pairings. More precisely,
\[
        T(c_{j,r})=c_{-j,p-2-r},
        \qquad
        R(c_{j,r})=c_{\lambda-j-r,r}.
\]
Furthermore, \(T\) reverses crossing signs and \(R\) preserves crossing signs.
\end{lemma}

\begin{proof}
The chord corresponding to \(c_{j,r}\) has endpoints \(u_{j,r}\) and
\(v_{j,r}\). The endpoint formulas give the stated action on crossings.

Since \(\varepsilon(c_{j,r})=(-1)^r\) and \(p-2\) is odd,
\[
        \varepsilon(T(c_{j,r}))
        =
        (-1)^{p-2-r}
        =
        -(-1)^r
        =
        -\varepsilon(c_{j,r}).
\]
On the other hand,
\[
        \varepsilon(R(c_{j,r}))
        =
        (-1)^r
        =
        \varepsilon(c_{j,r}).
\]
\end{proof}

\begin{lemma}\label{lem:over-under}
Both \(R\) and \(T\) exchange over-passing and under-passing preimages at every
crossing.
\end{lemma}

\begin{proof}
With the standard braid convention, at \(c_{j,r}\) the \(u\)-branch is
over-passing when \(r\) is even, and the \(v\)-branch is over-passing when
\(r\) is odd.

The map \(T\) sends \(r\) to \(p-2-r\). Since \(p-2\) is odd, \(r\) and
\(p-2-r\) have opposite parity. Since \(T\) preserves the \(u/v\)-label, it
exchanges over-passing and under-passing preimages.

The map \(R\) preserves \(r\) but exchanges \(u\) and \(v\). Hence \(R\) also
exchanges over-passing and under-passing preimages.
\end{proof}
 
\section{The rotation data}\label{sec:rotation_system}

For a crossing \(c=c_{j,r}\), write the incident half-edges as
\[
        u_c^-,\quad u_c^+,\quad v_c^-,\quad v_c^+,
\]
where \(u_c^-\) is the half-edge immediately before \(u_{j,r}\) on the
oriented Gauss circle, \(u_c^+\) is the half-edge immediately after
\(u_{j,r}\), and similarly for \(v_c^\pm\).

For the underlying immersed \(4\)-valent graph of the standard braid diagram,
use the cyclic order
\[
        \Omega_c=(u_c^-,\,v_c^+,\,u_c^+,\,v_c^-).
\]
The crossing sign and arrow data are recorded separately by the over/under
assignment and by the orientation of the Gauss circle.

\begin{lemma}\label{lem:rotation}
The map \(T\) preserves the cyclic order data of the associated abstract link
diagram, while \(R\) reverses it.
\end{lemma}

\begin{proof}
By Lemma~\ref{lem:reverse-circle}, both \(T\) and \(R\) reverse the orientation
of the Gauss circle. Hence, for \(F\in\{T,R\}\) and every Gauss endpoint \(e\),
\[
        F(e^-)=F(e)^+,\qquad F(e^+)=F(e)^-.
\]

Let \(c'=T(c)\). Since \(T\) preserves the \(u/v\)-label,
\[
\begin{aligned}
        T(\Omega_c)
        &=
        T(u_c^-,\,v_c^+,\,u_c^+,\,v_c^-)\\
        &=
        (u_{c'}^+,\,v_{c'}^-,\,u_{c'}^-,\,v_{c'}^+).
\end{aligned}
\]
This is a cyclic rotation of
\[
        \Omega_{c'}
        =
        (u_{c'}^-,\,v_{c'}^+,\,u_{c'}^+,\,v_{c'}^-).
\]
Thus \(T\) preserves the rotation data.

Let \(c''=R(c)\). Since \(R\) exchanges the \(u/v\)-label,
\[
\begin{aligned}
        R(\Omega_c)
        &=
        R(u_c^-,\,v_c^+,\,u_c^+,\,v_c^-)\\
        &=
        (v_{c''}^+,\,u_{c''}^-,\,v_{c''}^-,\,u_{c''}^+).
\end{aligned}
\]
The reverse cyclic order of
\[
        \Omega_{c''}
        =
        (u_{c''}^-,\,v_{c''}^+,\,u_{c''}^+,\,v_{c''}^-)
\]
is
\[
        (u_{c''}^-,\,v_{c''}^-,\,u_{c''}^+,\,v_{c''}^+),
\]
and \(R(\Omega_c)\) is a cyclic rotation of this reverse order. Therefore
\(R\) reverses the rotation data.
\end{proof}
  
\section{Supporting-surface realization}\label{sec:Realization}

Let \(\mathcal G_{p,q}\) be the signed, arrowed Gauss diagram described above.
Following Kamada--Kamada~\cite{KamadaKamada2000}, we regard these data as
defining an abstract link diagram. Namely, identifying the two endpoints of
each chord gives an abstract \(4\)-valent graph \(\Gamma_{p,q}\), and the
cyclic orders \(\Omega_c\) give the rotation data at the vertices.
Equivalently, these data prescribe how small neighbourhoods of the vertices are
attached to the neighbouring edge-bands.

This disk-band and capping construction is the supporting-surface realization
associated with the abstract link diagram; see also
Carter--Kamada--Saito~\cite{CKS2002} for the related stable-equivalence
framework for knots on surfaces and virtual knots.

Thus the Gauss data together with the verified rotation data determine a
disk-band surface, which we denote by
\[
        N(\Gamma_{p,q}).
\]
Concretely, \(N(\Gamma_{p,q})\) is obtained by replacing each crossing vertex
of \(\Gamma_{p,q}\) by a small disk and each edge of \(\Gamma_{p,q}\) by a
rectangular band attached to the corresponding boundary intervals of the
crossing disks.  The small crossing disks are part of the disk-band surface
itself; they should not be confused with the capping disks introduced below.
The surface \(N(\Gamma_{p,q})\) deformation retracts onto \(\Gamma_{p,q}\), and
its boundary is a disjoint union of circles
\[
        \partial N(\Gamma_{p,q})=C_1\sqcup\cdots\sqcup C_m.
\]
Each \(C_i\) is a boundary component of the entire disk-band surface, not the
boundary of a single band. 
The closed supporting surface \(\Sigma(\mathcal G_{p,q})\) is obtained by
capping these boundary components. That is, for each boundary circle \(C_i\)
one attaches a new disk \(D_{C_i}\) along its boundary:
\[
        \partial D_{C_i}=C_i.
\]
Therefore
\[
        \Sigma(\mathcal G_{p,q})
        =
        N(\Gamma_{p,q})
        \cup_{C_1}D_{C_1}
        \cup\cdots\cup_{C_m}D_{C_m}.
\]
Since \(\mathcal G_{p,q}\) is equipped with the rotation data inherited from
the standard classical closed-braid diagram on \(S^2\), this capping procedure
recovers the original sphere:
\[
        \Sigma(\mathcal G_{p,q})\cong S^2.
\]
Thus we are not using only the underlying chord diagram; the supporting surface
is determined by the signed, arrowed Gauss data together with the verified
rotation data.

\begin{figure}[htbp]
\vspace{-0.4\baselineskip}

\begin{subfigure}[t]{0.35\textwidth}
\centering
\includegraphics[width=\linewidth]{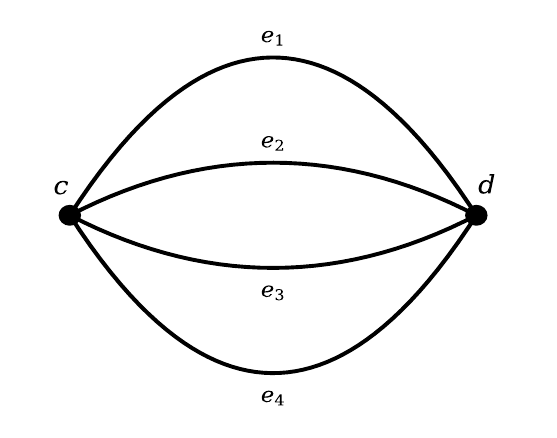}
\caption{Projected \(4\)-valent graph \(\Gamma\).}
\end{subfigure}
\hspace{0.06\textwidth}
\begin{subfigure}[t]{0.35\textwidth}
\centering
\includegraphics[width=\linewidth]{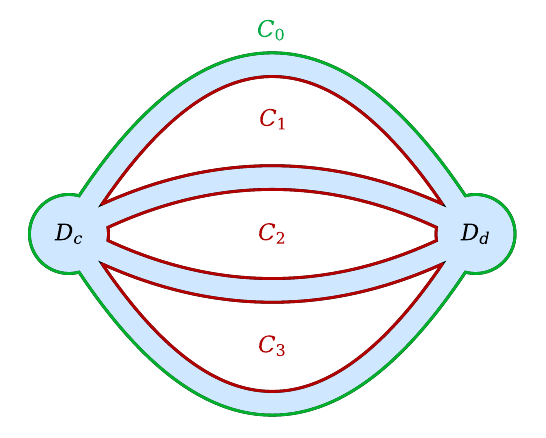}
\caption{The disk-band surface \(N(\Gamma)\) with boundary components.}
\end{subfigure}

\medskip

\begin{subfigure}[t]{0.35\textwidth}
\centering
\includegraphics[width=\linewidth]{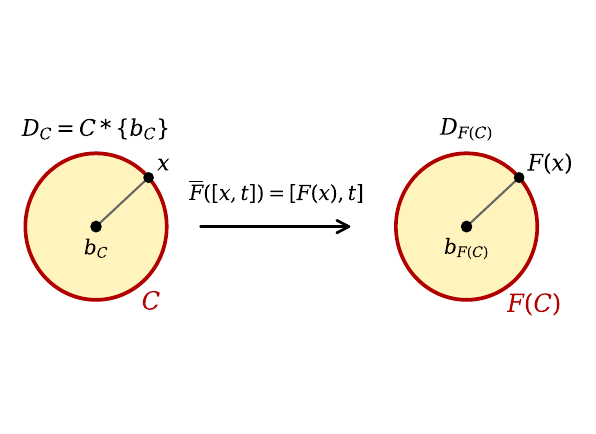}
\caption{Equivariant capping of one boundary component \(C\).}
\end{subfigure}
\hspace{0.06\textwidth}
\begin{subfigure}[t]{0.35\textwidth}
\centering
\includegraphics[width=\linewidth]{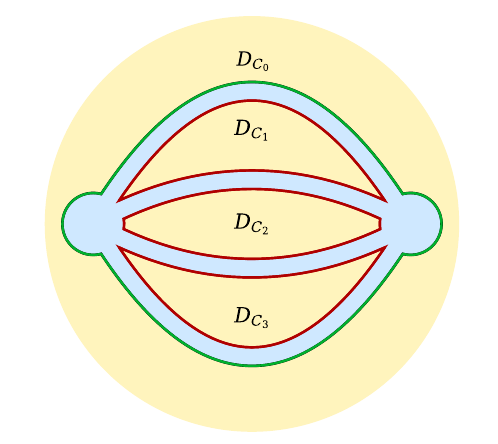}
\caption{Capping all boundary components to obtain the closed supporting surface.}
\end{subfigure}

\caption{A schematic model of the disk-band construction and capping step.
A projected \(4\)-valent graph \(\Gamma\) is thickened to a disk-band surface
\(N(\Gamma)\) by replacing vertices with crossing disks and edges with bands.
In this schematic, $\partial N(\Gamma)=C_0\sqcup C_1\sqcup C_2\sqcup C_3$,
where the green curve is the outer boundary component \(C_0\), and the red
curves are \(C_1,C_2,C_3\). The boundary components are capped by new disks
\(D_{C_i}\). In the classical case considered in this paper, these caps are
the complementary face disks, and the resulting closed supporting surface is
\(S^2\). }
\label{fig:disk-band-capping}
\vspace{-0.4\baselineskip}
\end{figure}

\FloatBarrier

We now show that the symmetries
\(R\) and \(T\) are realized by commuting homeomorphisms of this sphere.

\begin{lemma}\label{lem:realization}
The \(V_4\)-action generated by \(R\) and \(T\) is realized by commuting PL
homeomorphisms
\[
        h_R,h_T:S^2\longrightarrow S^2.
\]
These maps preserve the underlying immersed \(4\)-valent diagram and induce on
its Gauss data the actions \(R\) and \(T\). Moreover, \(h_R\) is
orientation-reversing and \(h_T\) is orientation-preserving.
\end{lemma}

\begin{proof}
By Lemmas~\ref{lem:reverse-circle} and~\ref{lem:crossings}, the maps
\(R\) and \(T\) preserve both the chord pairings and the adjacency of
successive Gauss endpoints. Hence they induce automorphisms of
\(\Gamma_{p,q}\). By Lemma~\ref{lem:rotation}, \(T\) preserves the rotation
data and \(R\) reverses it.

Thicken \(\Gamma_{p,q}\) according to this rotation data. Since \(T\)
preserves the cyclic order at every crossing disk, it induces an
orientation-preserving PL homeomorphism of the disk-band surface
\(N(\Gamma_{p,q})\). Since \(R\) reverses the cyclic order at every crossing
disk, it induces an orientation-reversing PL homeomorphism of
\(N(\Gamma_{p,q})\). Choose the crossing disks and edge-bands equivariantly with respect to
the finite group \(\langle R,T\rangle\). On every crossing disk and
edge-band, extend the induced permutation of the marked attaching
intervals by the corresponding affine PL map. This construction is
functorial in the graph automorphism; consequently it realizes the
entire \(V_4\)-action on \(N(\Gamma_{p,q})\), rather than merely
realizing the two generators separately. In particular, the resulting
homeomorphisms commute. Since these maps are homeomorphisms of \(N(\Gamma_{p,q})\), they permute the
boundary components of \(N(\Gamma_{p,q})\).

It remains to cap the boundary components equivariantly. For each boundary
component \(C\subset\partial N(\Gamma_{p,q})\), attach a disk \(D_C\). We view
\(D_C\) as the cone
\[
        D_C=C*\{b_C\},
\]
where \(b_C\) is the cone point. If \(F\in\langle R,T\rangle\) sends a boundary
component \(C\) to a boundary component \(F(C)\), define the extension over
the corresponding capping disk by
\[
        \overline F([x,t])=[F(x),t]\in D_{F(C)},
        \qquad x\in C,\quad 0\leq t\leq 1.
\]
Here \(t=1\) corresponds to the boundary and \(t=0\) to the cone point. Thus
the action on each boundary circle is extended radially over the corresponding
capping disk; in particular, if \(F(C)=C'\), then the cone point \(b_C\) is sent
to \(b_{C'}\). This extension is PL after choosing compatible triangulations,
and it is well-defined also when \(C\) is setwise fixed by a nontrivial element.

The extensions satisfy
\[
        \overline{FG}=\overline F\,\overline G
\]
on every capping disk, because both sides send \([x,t]\) to
\([FG(x),t]\). Thus the \(V_4\)-action extends over the capped surface without
changing the group law. Since the capped supporting surface is \(S^2\), we obtain
commuting PL homeomorphisms
\[
        h_R,h_T:S^2\to S^2.
\]
Moreover \(h_T\) is orientation-preserving and \(h_R\) is
orientation-reversing, according as the corresponding automorphism preserves
or reverses the rotation data.
\end{proof}

\begin{remark}
The capping disks \(D_C\) are distinct from the small crossing disks used to
construct \(N(\Gamma_{p,q})\). The crossing disks are part of the disk-band
neighbourhood of the projected graph, whereas the disks \(D_C\) are attached
afterwards to fill the boundary components of that disk-band surface. In the
classical case considered here, these capping disks are precisely the face
disks needed to recover the sphere \(S^2\).
\end{remark}

The surface maps constructed in Lemma~\ref{lem:realization} preserve
the underlying projected diagram. They do not by themselves preserve the full
over/under-decorated diagram. Instead, Lemma~\ref{lem:over-under} shows that
they exchange over- and under-passing branches, and this exchange will be
compensated by the normal reversal in the \(S^3\)-thickening.
  
\begin{lemma}\label{lem:thickening}
The surface maps \(h_R,h_T\) thicken to commuting ambient involutions
\[
        \widehat{R},\widehat{T}:S^3\longrightarrow S^3
\]
preserving \(K_{p,q}\), where \(\widehat{R}\) is orientation-preserving and
\(\widehat{T}\) is orientation-reversing.
\end{lemma}

\begin{proof}

Since \(\Sigma(\mathcal G_{p,q})\cong S^2\), realize \(S^3\) as its suspension:
\[
S^3 \cong \Sigma S^2
= \bigl(S^2\times[-1,1]\bigr)\big/\sim,
\]
where \(\sim\) is the equivalence relation defined by
\[
(x,-1)\sim(y,-1),\qquad (x,1)\sim(y,1)
\quad\text{for all }x,y\in S^2.
\]
Denote the two corresponding equivalence classes by \(S_-\) and \(S_+\).

Let
\[
        \gamma:S^1\longrightarrow S^2
\]
denote the immersion underlying the knot diagram. 

Let
\[
        G=\langle R,T\rangle\cong \mathbb Z_2\oplus\mathbb Z_2,
\]
and define the character
\[
        \chi:G\longrightarrow\{\pm1\}
\]
by
\[
        \chi(R)=\chi(T)=-1.
\]
Thus
\[
        \chi(1)=\chi(RT)=1.
\]

Prescribe the value \(+\delta\) at every over-passing marked preimage and
the value \(-\delta\) at every under-passing marked preimage. By
Lemma~\ref{lem:over-under}, these prescribed values are
\(\chi\)-equivariant: for every marked preimage \(e\in E_{p,q}\) and every
\(g\in G\), the prescribed value at \(g(e)\) is \(\chi(g)\) times the
prescribed value at \(e\).

Choose any sufficiently small piecewise-linear function
\[
        \eta_0:S^1\longrightarrow(-1,1)
\]
taking these prescribed values at all marked preimages, and define
\[
        \eta(z)
        =
        \frac{1}{4}\sum_{g\in G}\chi(g)\eta_0(gz).
\]
Then \(\eta\) is piecewise linear. Moreover, for \(h\in G\),
\[
\begin{aligned}
        \eta(hz)
        &=
        \frac{1}{4}\sum_{g\in G}\chi(g)\eta_0(ghz)\\
        &=
        \frac{1}{4}\sum_{k\in G}\chi(kh^{-1})\eta_0(kz)\\
        &=
        \chi(h)\eta(z),
\end{aligned}
\]
where we used that \(G\) is abelian and \(\chi(h^{-1})=\chi(h)\).
In particular,
\[
        \eta(Rz)=-\eta(z),
        \qquad
        \eta(Tz)=-\eta(z).
\]

The averaging does not alter the prescribed values at the marked
preimages. Indeed, if \(e\in E_{p,q}\) has prescribed value \(a_e\), then
the prescribed value at \(g(e)\) is \(\chi(g)a_e\), and hence
\[
        \eta(e)
        =
        \frac{1}{4}\sum_{g\in G}\chi(g)\eta_0(g(e))
        =
        \frac{1}{4}\sum_{g\in G}\chi(g)^2a_e
        =
        a_e.
\]
Thus \(\eta\) takes the values \(+\delta\) and \(-\delta\) on the
over-passing and under-passing marked preimages, respectively. Because each summand \(\chi(g)\eta_0(gz)\) lies in \((-1,1)\) and
\((-1,1)\) is convex, one has
\[
        \eta(S^1)\subset(-1,1).
\]. Finally, if \(z\) is fixed by \(R\) or by \(T\), the corresponding
equivariance identity forces
\[
        \eta(z)=0.
\]

The lifted knot is the image of the embedding
\[
        \widetilde{\gamma}:S^1\longrightarrow \Sigma S^2,
        \qquad
        \widetilde{\gamma}(z)=[\gamma(z),\eta(z)].
\]
It is embedded because the only double points of \(\gamma\) are the prescribed
crossings, and their two preimages are separated by the distinct heights
\(+\delta\) and \(-\delta\).

For \(-1<s<1\), define
\[
\widehat{R}(x,s)=(h_R(x),-s),\qquad
\widehat{T}(x,s)=(h_T(x),-s).
\]
At the suspension points, set
\[
        \widehat{R}(S_+)=S_-,
        \qquad
        \widehat{R}(S_-)=S_+,
\]
and similarly
\[
        \widehat{T}(S_+)=S_-,
        \qquad
        \widehat{T}(S_-)=S_+.
\]
These formulas are compatible with the suspension quotient, since both maps
send \(S^2\times\{1\}\) to \(S^2\times\{-1\}\) and conversely. In particular,
no suspension point is fixed by either map.

Since \(h_R\) and \(h_T\) commute, and since both lifted maps reverse the normal
coordinate, the maps \(\widehat{R}\) and \(\widehat{T}\) are commuting
involutions. The map \(h_R\) reverses the orientation of \(S^2\), and the
coordinate change \(s\mapsto -s\) reverses the normal direction. Thus
\(\widehat{R}\) is orientation-preserving on \(S^3\). The map \(h_T\) preserves
the orientation of \(S^2\), so \(\widehat{T}\) is orientation-reversing on
\(S^3\).

For \(F\in\{R,T\}\), the realization on the supporting sphere satisfies
\[
        h_F\circ\gamma=\gamma\circ F.
\]
Moreover, the choice of the height function gives
\[
        \eta(Fz)=-\eta(z).
\]
Consequently,
\[
\begin{aligned}
        \widehat F\bigl(\widetilde{\gamma}(z)\bigr)
        &=
        \widehat F[\gamma(z),\eta(z)] \\
        &=
        [h_F(\gamma(z)),-\eta(z)] \\
        &=
        [\gamma(Fz),\eta(Fz)] \\
        &=
        \widetilde{\gamma}(Fz).
\end{aligned}
\]
Thus both \(\widehat R\) and \(\widehat T\) preserve the embedded knot
\(K_{p,q}\).
\end{proof}

\begin{lemma}\label{lem:fixed-points-on-knot}
For \(F\in\{R,T\}\), one has
\[
        \Fix(\widehat F)\cap K_{p,q}
        =
        \widetilde{\gamma}\bigl(\Fix(F)\bigr).
\]
\end{lemma}

\begin{proof}
If \(Fz=z\), equivariance gives
\[
        \widehat F\bigl(\widetilde{\gamma}(z)\bigr)
        =
        \widetilde{\gamma}(Fz)
        =
        \widetilde{\gamma}(z).
\]
Conversely, suppose
\(\widehat F(\widetilde{\gamma}(z))=\widetilde{\gamma}(z)\).
Since
\[
        \widehat F\circ\widetilde{\gamma}
        =
        \widetilde{\gamma}\circ F,
\]
we have
\[
        \widetilde{\gamma}(Fz)=\widetilde{\gamma}(z).
\]
The map \(\widetilde{\gamma}\) is an embedding, and hence is injective.
Therefore \(Fz=z\).
\end{proof}

\begin{lemma}\label{lem:equivariant-smoothing}
After a \(V_4\)-equivariant PL isotopy of \(K_{p,q}\), supported away from
\[
        K_{p,q}\cap
        \bigl(\Fix(\widehat R)\cup\Fix(\widehat T)\bigr),
\]
the \(V_4\)-action on the PL pair \((S^3,K_{p,q})\) constructed above can be
realized by a smooth \(V_4\)-action on the standard smooth \(3\)-sphere.
In particular, the fixed-set intersection properties established below are
unchanged.
\end{lemma}

\begin{proof}
Choose a \(V_4\)-invariant triangulation of \(S^3\) for which
\(K_{p,q}\) is an invariant one-dimensional subcomplex. By the
equivariant smoothing theorem of Lange~\cite{LangeSmoothing}, this PL
\(V_4\)-action admits a compatible equivariant smooth structure. The finitely many corners of the invariant polygonal knot may be rounded
simultaneously on their \(V_4\)-orbits. Since the points of
\[
        K_{p,q}\cap
        \bigl(\Fix(\widehat R)\cup\Fix(\widehat T)\bigr)
\]
lie in interiors of diagram arcs, the rounding neighbourhoods may be chosen
disjoint from these points. The resulting equivariant PL isotopy therefore
preserves the fixed-set intersections and yields a smooth \(V_4\)-invariant
knot in the same PL isotopy class. Finally, the
smooth structure on \(S^3\) is diffeomorphic to the standard one, so
the action may be transported to the standard smooth \(3\)-sphere.
\end{proof}

\section{Fixed sets}\label{sec:fixed_sets}

It remains to identify the fixed sets. By Kerékjártó's classification of
periodic homeomorphisms of the sphere, in the modern formulation of
Constantin--Kolev~\cite[Theorem~4.1]{ConstantinKolev1994}, every periodic
homeomorphism of \(S^2\) is topologically conjugate to an orthogonal
transformation. Consequently, a nontrivial orientation-preserving involution
of \(S^2\) has fixed set homeomorphic to \(S^0\), while an
orientation-reversing involution of \(S^2\) with nonempty fixed set has fixed
set homeomorphic to \(S^1\).
  
\begin{lemma}\label{lem:no-endpoints}
Neither \(R\) nor \(T\) fixes a marked endpoint of the Gauss diagram.
Consequently, the restrictions of \(R\) and \(T\) to the knot circle have
exactly two fixed points each, all lying in interiors of diagram arcs.
\end{lemma}

\begin{proof}
For \(T\), a fixed marked endpoint would require
\[
        r=p-2-r,
\]
equivalently
\[
        2r=p-2.
\]
Since \(p-2\) is odd, this has no integer solution. For \(R\), fixed marked
endpoints are impossible because \(R\) exchanges \(u\)- and \(v\)-endpoints.

By Lemma~\ref{lem:reverse-circle}, both endpoint maps reverse the cyclic order
defined by the successor map. By the extension described above, they are orientation-reversing PL
involutions of the knot circle. An orientation-reversing involution of a circle has exactly
two fixed points. Since no marked endpoint is fixed, the fixed points lie in
interiors of diagram arcs. 
\end{proof}
  
\begin{lemma}\label{lem:fixed-R}
The involution \(\widehat{R}\) is a strong inversion of \(K_{p,q}\).
Moreover, Lemma~\ref{lem:reverse-circle} shows that
\(\widehat R|_{K_{p,q}}\) reverses the orientation of the knot.
\end{lemma}

\begin{proof}
The surface map \(h_R\) is an orientation-reversing involution of \(S^2\), and
it has fixed points because \(R\) has fixed points on the knot circle. Hence
\[
        \Fix(h_R)\cong S^1.
\]
Since
\[
        \widehat{R}(x,s)=(h_R(x),-s),
\]
a point is fixed by \(\widehat{R}\) if and only if \(s=0\) and
\(x\in\Fix(h_R)\). Thus
\[
        \Fix(\widehat{R})=\Fix(h_R)\cong S^1.
\]

By Lemma~\ref{lem:no-endpoints}, the restriction of \(R\) to the knot circle
has exactly two fixed points. Lemma~\ref{lem:fixed-points-on-knot} therefore
gives
\[
        \#\bigl(K_{p,q}\cap\Fix(\widehat R)\bigr)
        =
        \#\Fix(R)
        =
        2.
\]
Thus
\[
        \Fix(\widehat R)\cong S^1,
        \qquad
        \#\bigl(K_{p,q}\cap\Fix(\widehat R)\bigr)=2.
\]
Since \(\widehat{R}\) is orientation-preserving by
Lemma~\ref{lem:thickening}, it is a strong inversion.
\end{proof}
  
\begin{lemma}\label{lem:fixed-T}
The involution \(\widehat{T}\) is a strongly negative amphichiral involution
for \(K_{p,q}\); in particular, its restriction to \(K_{p,q}\) reverses the
orientation of the knot.
\end{lemma}

\begin{proof}
The surface map \(h_T\) is a nontrivial orientation-preserving involution of
\(S^2\). Hence
\[
        \Fix(h_T)\cong S^0.
\]
Since
\[
        \widehat{T}(x,s)=(h_T(x),-s),
\]
one has
\[
        \Fix(\widehat{T})=\Fix(h_T)\cong S^0.
\]

By Lemma~\ref{lem:no-endpoints}, the restriction of \(T\) to the knot circle
has exactly two fixed points. Lemma~\ref{lem:fixed-points-on-knot} therefore
gives
\[
        \#\bigl(K_{p,q}\cap\Fix(\widehat T)\bigr)
        =
        \#\Fix(T)
        =
        2.
\]
On the other hand,
\[
        \Fix(\widehat T)\cong S^0
\]
also consists of exactly two points. Since
\[
        K_{p,q}\cap\Fix(\widehat T)
        \subseteq
        \Fix(\widehat T),
\]
it follows that
\[
        K_{p,q}\cap\Fix(\widehat T)
        =
        \Fix(\widehat T).
\]
By Lemma~\ref{lem:thickening}, \(\widehat{T}\) is orientation-reversing.
Hence \(\widehat{T}\) is a strongly negative amphichiral involution for
\(K_{p,q}\).
\end{proof}

\begin{proof}[Proof of Theorem~\ref{thm:odd-klein}]
Set
\[
        \rho=\widehat{R},\qquad \tau=\widehat{T}.
\]
By Lemma~\ref{lem:thickening}, \(\rho\) and \(\tau\) are commuting ambient
involutions of \(S^3\) preserving \(K_{p,q}\), with \(\rho\)
orientation-preserving and \(\tau\) orientation-reversing. By
Lemma~\ref{lem:fixed-R}, \(\rho\) is a strong inversion. By
Lemma~\ref{lem:fixed-T}, \(\tau\) is a strongly negative amphichiral involution. Hence
\(K_{p,q}\) is Klein amphichiral.

Moreover, the subgroup generated by \(\rho\) and \(\tau\) is a Klein four
group. Indeed, \(\rho^2=\tau^2=1\) and \(\rho\tau=\tau\rho\). Both involutions
are nontrivial, and \(\rho\neq\tau\) because \(\rho\) is orientation-preserving
whereas \(\tau\) is orientation-reversing. Hence
\[
        \langle \rho,\tau\rangle
        =
        \{1,\rho,\tau,\rho\tau\}
        \cong
        \mathbb Z_2\oplus\mathbb Z_2.
\]
The preceding argument constructs the action in the PL category.
By Lemma~\ref{lem:equivariant-smoothing}, after an equivariant PL isotopy
of \(K_{p,q}\), this action is smooth with respect to a compatible smooth
structure on \(S^3\). Transporting that structure to the standard smooth
\(3\)-sphere gives commuting smooth involutions with the same orientation
and fixed-set properties. Since the equivariantly isotoped knot represents
the same knot type, we retain the notation \(K_{p,q}\), \(\rho\), and \(\tau\).
\end{proof}

The schematic Figure~\ref{fig:v4-schematic} illustrates the subgroup constructed above at the level
of the supporting-surface realization.
For related visual models of these symmetries in Turk's head knots, see
Di Prisa--\c{S}avk~\cite[Figures~2 and~8]{DPSequivariant}.

\begin{figure}[htbp]
\centering
\begin{tikzpicture}[scale=0.9, >=Latex]

\def\gap{0.16}



\draw[thick] (-2.75,1.45) -- (-2.75,2.35) -- (2.75,2.35) -- (2.75,1.45);

\draw[thick] (-2.35,1.45) -- (-2.35,2.05) -- (2.35,2.05) -- (2.35,1.45);

\draw[thick] (-1.95,1.45) -- (-1.95,1.75) -- (-\gap,1.75);
\draw[thick] (\gap,1.75) -- (1.95,1.75) -- (1.95,1.45);


\draw[thick] (-2.75,-1.45) -- (-2.75,-2.35) -- (2.75,-2.35) -- (2.75,-1.45);

\draw[thick] (-2.35,-1.45) -- (-2.35,-2.05) -- (2.35,-2.05) -- (2.35,-1.45);

\draw[thick] (-1.95,-1.45) -- (-1.95,-1.75) -- (-\gap,-1.75);
\draw[thick] (\gap,-1.75) -- (1.95,-1.75) -- (1.95,-1.45);

\foreach \y/\xout in {1.25/-3.55,1.10/-3.80,0.95/-4.05}{
  \draw[thick]
    (-3.15,\y) -- (\xout,\y) -- (\xout,-\y) -- (-3.15,-\y);
}

\foreach \y/\xout in {1.25/3.55,1.10/3.80,0.95/4.05}{
  \draw[thick]
    (3.15,\y) -- (\xout,\y) -- (\xout,-\y) -- (3.15,-\y);
}


\draw[red!75!black, thick, dashed] (0,-2.55) -- (0,2.55);

\draw[white, line width=5pt] (0,2.35-\gap) -- (0,2.35+\gap);
\draw[white, line width=5pt] (0,-2.35-\gap) -- (0,-2.35+\gap);

\draw[thick] (-0.22,2.35) -- (0.22,2.35);
\draw[thick] (-0.22,-2.35) -- (0.22,-2.35);

\fill[red!75!black] (0,2.05) circle (2.2pt);
\fill[red!75!black] (0,-2.05) circle (2.2pt);


\draw[thick, fill=white, rounded corners] (-3.15,0.75) rectangle (-1.55,1.45);
\node at (-2.35,1.10) {\(\rho(B)\)};

\draw[thick, fill=white, rounded corners] (1.55,0.75) rectangle (3.15,1.45);
\node at (2.35,1.10) {\(B\)};

\draw[thick, fill=white, rounded corners] (-3.15,-1.45) rectangle (-1.55,-0.75);
\node at (-2.35,-1.10) {\(\rho\tau(B)\)};

\draw[thick, fill=white, rounded corners] (1.55,-1.45) rectangle (3.15,-0.75);
\node at (2.35,-1.10) {\(\tau(B)\)};


\fill[blue!80!black] (-3.80,0) circle (2.8pt);
\fill[blue!80!black] (3.80,0) circle (2.8pt);


\foreach \x in {-4.55,-4.45,-4.35,-4.25}{
  \fill[black!45] (\x,0) circle (0.55pt);
}
\foreach \x in {4.25,4.35,4.45,4.55}{
  \fill[black!45] (\x,0) circle (0.55pt);
}

\draw[red!75!black, thick, ->] (1.55,2.85) -- (-1.55,2.85);
\node[red!75!black, above] at (0,2.85) {\(\rho\)};

\draw[blue!80!black, thick, ->] (4.55,1.05) -- (4.55,-1.05);
\node[blue!80!black, right] at (4.65,0) {\(\tau\)};

\draw[green!60!black, thick, ->, bend right=18]
  (1.35,0.50) to (-1.35,-0.50);
\node[green!60!black] at (0.55,0.18) {\(\rho\tau\)};

\end{tikzpicture}

\caption{
A schematic block diagram illustrating the \(V_4\)-symmetry constructed in
Theorem~\ref{thm:odd-klein}. A fundamental block \(B\) is accompanied by its
images \(\rho(B)\), \(\tau(B)\), and \(\rho\tau(B)\). The red dashed line
schematically indicates the fixed axis of the strong inversion \(\rho\), and
the two red points indicate its two intersections with the knot in this
schematic model. Near each red point, the two adjacent strands are drawn on
opposite sides of the \(\rho\)-axis in the projection, in a \(\tau\)-symmetric
way. The two blue points indicate the two points of
\(\operatorname{Fix}(\tau)\cap K\). The gray dotted pieces indicate
continuation of the strand bundles in the general case, and should not be
interpreted as part of the \(\rho\)-axis. The picture records the orbit of a
block under the commuting involutions and should not be read as a standard
planar projection of the knot.
}
\label{fig:v4-schematic}
\end{figure}
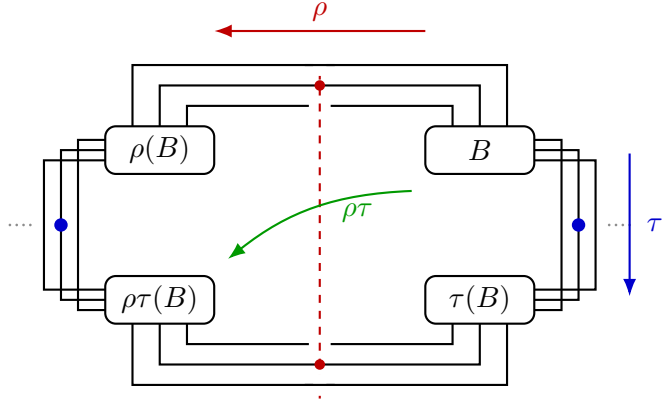

\begin{remark}\label{rem:full-symmetry-group}
Let \(\operatorname{Sym}(S^3,K)\) denote the group of isotopy classes of
self-homeomorphisms of the pair \((S^3,K)\). The construction above gives
an explicit subgroup
\[
        V_4=\langle\rho,\tau\rangle
        \leq \operatorname{Sym}(S^3,\Th(p,q)).
\]
Indeed, the four resulting mapping classes are distinct: the identity
preserves both the ambient orientation and the knot orientation; \(\rho\)
preserves the ambient orientation and reverses the knot orientation;
\(\tau\) reverses both; and \(\rho\tau\) reverses the ambient orientation
while preserving the knot orientation.

For the three-stranded family with \(\gcd(3,q)=1\), Sakuma--Weeks~\cite{SakumaWeeks1995} computed
the full symmetry group; with the convention that \(D_n\) denotes the dihedral
group of order \(2n\), one has
\[
        \operatorname{Sym}(S^3,\Th(3,q))\cong D_{2q}.
\]
For general coprime Turk's head knots, Di Prisa--\c{S}avk conjecture that
\[
        \operatorname{Sym}(S^3,\Th(p,q))\cong D_{2q}.
\]
See~\cite{DPSsurvey}. Determining the full symmetry group in general and
identifying the subgroup constructed here within it would be a natural
refinement of the present work.
\end{remark}

\section{Laurent squares and the Burau formula}\label{sec:laurent-prelim}

\subsection{Squares in Laurent polynomial rings}

Let \(R=\ZZ[t,t^{-1}]\).  We say that a Laurent polynomial \(F(t)\in R\) is a \emph{Laurent square up to unit} if
\[
        F(t)\doteqeq G(t)^2
\]
for some \(G(t)\in R\).  This is the relevant notion because the Alexander polynomial is defined only up to multiplication by \(\pm t^r\).

We shall use the following elementary observation.

\begin{lemma}\label{lem:F2-square-parity}
Let \(0\neq F(t)\in\mathbb F_2[t,t^{-1}]\).  If \(F(t)\doteqeq G(t)^2\) for some \(G(t)\in\mathbb F_2[t,t^{-1}]\), then all exponents appearing in \(F(t)\) lie in a single congruence class modulo \(2\).
\end{lemma}

\begin{proof}
Write
\[
        G(t)=\sum_i a_i t^i,
        \qquad a_i\in\mathbb F_2.
\]
In characteristic \(2\), the Frobenius homomorphism gives
\[
        G(t)^2=\sum_i a_i^2 t^{2i}=\sum_i a_i t^{2i}.
\]
Thus every exponent in \(G(t)^2\) is even.  Multiplication by a Laurent monomial \(t^r\) shifts all exponents by \(r\), so all exponents of \(t^rG(t)^2\) have the same parity.  The sign is irrelevant over \(\mathbb F_2\).
\end{proof}

\subsection{Reduced Burau formula}

Let \(\Bur_t:B_p\to GL_{p-1}(\ZZ[t,t^{-1}])\) denote the reduced Burau representation.  We use the standard braid-closure formula
\begin{equation}\label{eq:burau-alexander}
        \Delta_{\widehat \beta}(t)
        \doteq
        \frac{\det(I-\Bur_t(\beta))}
             {1+t+\cdots+t^{p-1}}
\end{equation}
for the closure of a \(p\)-braid \(\beta\)~\cite{Birman}. Different standard conventions for the reduced Burau representation alter the right side only by a Laurent unit; this is harmless for all square/non-square arguments below.

We shall also use
\begin{equation}\label{eq:burau-det-generator}
        \det\Bur_t(\sigma_i)=-t.
\end{equation}

\section{The Burau square quotient}\label{sec:burau-square-quotient}

In this section we prove the first main algebraic input.  Let
\[
        M_p(t)=\Bur_t(A_p).
\]

\subsection{Garside conjugacy}

Let \(\Dhalf_p\in B_p\) denote the Garside half-twist. We use the standard
relation~\cite{KasselTuraev2008}

\begin{equation}\label{eq:garside-reversal}
        \Dhalf_p\sigma_i\Dhalf_p^{-1}=\sigma_{p-i}.
\end{equation}

\begin{lemma}\label{lem:garside-A-inverse}
If \(p\) is odd, then
\[
        \Dhalf_p A_p\Dhalf_p^{-1}=A_p^{-1}.
\]
Consequently \(M_p(t)\) is conjugate over \(\ZZ[t,t^{-1}]\) to \(M_p(t)^{-1}\).
\end{lemma}

\begin{proof}
Using \eqref{eq:garside-reversal}, we compute
\[
\begin{aligned}
        \Dhalf_p A_p\Dhalf_p^{-1}
        & = \sigma_{p-1}\sigma_{p-2}^{-1}\sigma_{p-3}\cdots
            \sigma_2\sigma_1^{-1}.
\end{aligned}
\]
Since \(p\) is odd, \(p-1\) is even, and
\[
        A_p=\sigma_1\sigma_2^{-1}\sigma_3\sigma_4^{-1}\cdots
        \sigma_{p-2}\sigma_{p-1}^{-1}.
\]
Therefore
\[
\begin{aligned}
        A_p^{-1}
        &= (\sigma_{p-1}^{-1})^{-1}(\sigma_{p-2})^{-1}
           \cdots (\sigma_2^{-1})^{-1}\sigma_1^{-1}       \\
        &= \sigma_{p-1}\sigma_{p-2}^{-1}\sigma_{p-3}\cdots
            \sigma_2\sigma_1^{-1}.
\end{aligned}
\]
Thus \(\Dhalf_pA_p\Dhalf_p^{-1}=A_p^{-1}\).  Applying the reduced Burau representation gives
\[
        \Bur_t(\Dhalf_p)M_p(t)\Bur_t(\Dhalf_p)^{-1}=M_p(t)^{-1}.
\]
\end{proof}

\begin{lemma}\label{lem:reciprocal-charpoly}
Let \(p\geq 3\) be odd and set \(p-1=2g\).  Then the characteristic polynomial
\[
        \chi_p(x)=\det(xI-M_p(t))
\]
is reciprocal:
\[
        \chi_p(x)=x^{2g}\chi_p(x^{-1}).
\]
\end{lemma}

\begin{proof}
By Lemma~\ref{lem:garside-A-inverse}, \(M_p(t)\) is conjugate to \(M_p(t)^{-1}\).  Hence
\[
        \chi_p(x)=\chi_{M_p}(x) = \chi_{M_p^{-1}}(x).
\]
The exponent sum of \(A_p\) is
\[
        1-1+1-1+\cdots+1-1=0,
\]
so by \eqref{eq:burau-det-generator},
\[
        \det M_p(t)=1.
\]
Since \(2g=p-1\) is even,
\[
\begin{aligned}
        \chi_{M_p^{-1}}(x)
        &=\det(xI-M_p^{-1})                                      \\
        &=\det(M_p^{-1})\det(xM_p-I)                              \\
        &=\det(xM_p-I)                                            \\
        &=\det(I-xM_p).
\end{aligned}
\]
On the other hand,
\[
\begin{aligned}
        x^{2g}\chi_p(x^{-1})
        &=x^{2g}\det(x^{-1}I-M_p) \\
        &=\det(I-xM_p).
\end{aligned}
\]
Thus \(\chi_{M_p^{-1}}(x)=x^{2g}\chi_p(x^{-1})\), and the claim follows.
\end{proof}

\subsection{The Chebyshev-resultant square}

Let
\[
        Q_m(X)=1+X+\cdots+X^{m-1}.
\]
We shall prove that \(\det Q_m(M_p(t)^2)\) is a Laurent square.

We shall use the following standard facts about reciprocal polynomials and
resultants; see, for example, Lang~\cite{Lang2002}.
If
\[
        \chi(x)=x^{2g}+a_1x^{2g-1}+\cdots+a_gx^g+\cdots+a_1x+1
\]
is monic reciprocal of degree \(2g\), then there is a unique monic polynomial \(\Psi(z)\) of degree \(g\) such that
\begin{equation}\label{eq:Psi-definition}
        \chi(x)=x^g\Psi(x+x^{-1}).
\end{equation}
This follows because the Laurent polynomials
\[
        x^k+x^{-k}\qquad (0\leq k\leq g)
\]
are integral polynomials in \(x+x^{-1}\).

Let \(S_r(z)\in\ZZ[z]\) be defined by
\begin{equation}\label{eq:S-recursion}
        S_0(z)=1,
        \qquad S_1(z)=z,
        \qquad S_{r+1}(z)=zS_r(z)-S_{r-1}(z).
\end{equation}
These are the Chebyshev polynomials of the second kind in the normalization
\begin{equation}\label{eq:S-lambda}
        S_{m-1}(\lambda+\lambda^{-1})
        =\frac{\lambda^m-\lambda^{-m}}{\lambda-\lambda^{-1}}.
\end{equation}
The identity is a polynomial identity in \(\lambda\) after clearing denominators, so it remains valid at \(\lambda=\pm 1\) by continuity or by the recursion.

\begin{proposition}\label{prop:quotient-square}
Let \(p\geq 3\) be odd and \(m\geq 1\).  Then
\[
        \det Q_m(M_p(t)^2)
\]
is a square in \(\ZZ[t,t^{-1}]\).
\end{proposition}

\begin{proof}
Set \(p-1=2g\).  By Lemma~\ref{lem:reciprocal-charpoly}, \(\chi_p(x)\) is monic reciprocal of degree \(2g\).  Let \(\Psi_p(z)\in\ZZ[t,t^{-1}][z]\) be the monic polynomial satisfying
\[
        \chi_p(x)=x^g\Psi_p(x+x^{-1}).
\]
Over an algebraic closure of \(\QQ(t)\), write the roots of \(\chi_p\) as
\[
        \lambda_1,\lambda_1^{-1},\ldots,\lambda_g,\lambda_g^{-1}.
\]
Then the roots of \(\Psi_p\) are
\[
        z_i=\lambda_i+\lambda_i^{-1},
        \qquad 1\leq i\leq g,
\]
with multiplicity.

Now
\[
\begin{aligned}
        \det Q_m(M_p^2)
        &=\prod_{i=1}^g Q_m(\lambda_i^2)Q_m(\lambda_i^{-2}).
\end{aligned}
\]
For an indeterminate \(\lambda\), one has
\begin{equation}\label{eq:Q-pair-square}
        Q_m(\lambda^2)Q_m(\lambda^{-2})
        =S_{m-1}(\lambda+\lambda^{-1})^2.
\end{equation}
Indeed, for \(\lambda^2\neq 1\),
\[
        Q_m(\lambda^2)=\frac{\lambda^{2m}-1}{\lambda^2-1},
\]
and a direct simplification gives
\[
        Q_m(\lambda^2)Q_m(\lambda^{-2})
        =\left(\frac{\lambda^m-\lambda^{-m}}
                    {\lambda-\lambda^{-1}}\right)^2.
\]
Together with \eqref{eq:S-lambda}, this proves \eqref{eq:Q-pair-square}; the exceptional values \(\lambda=\pm1\) follow because both sides are Laurent polynomial identities.

Thus
\[
\begin{aligned}
        \det Q_m(M_p^2)
        &=\prod_{i=1}^g S_{m-1}(z_i)^2 \\
        &=\left(\prod_{i=1}^g S_{m-1}(z_i)\right)^2.
\end{aligned}
\]
Since \(\Psi_p\) is monic,
\[
        \prod_{i=1}^g S_{m-1}(z_i)
        =\Res_z(\Psi_p(z),S_{m-1}(z)),
\]
up to the standard sign convention for resultants.  The sign is immaterial after squaring.  The resultant belongs to \(\ZZ[t,t^{-1}]\).  Hence
\[
        \det Q_m(M_p^2)
        =\Res_z(\Psi_p,S_{m-1})^2
\]
is a Laurent square.
\end{proof}

\begin{theorem}\label{thm:square-quotient}
Let \(p\geq 3\) be odd and \(m\geq 1\).  Then there exists \(H_{p,m}(t)\in\ZZ[t,t^{-1}]\) such that
\[
        \Delta_{\Th(p,2m)}(t)
        \doteq
        \Delta_{\Th(p,2)}(t)H_{p,m}(t)^2.
\]
\end{theorem}

\begin{proof}
Using \eqref{eq:burau-alexander} for \(\beta=A_p^{2m}\), we get
\[
        \Delta_{\Th(p,2m)}(t)
        \doteq
        \frac{\det(I-M_p^{2m})}{1+t+\cdots+t^{p-1}}.
\]
Similarly,
\[
        \Delta_{\Th(p,2)}(t)
        \doteq
        \frac{\det(I-M_p^2)}{1+t+\cdots+t^{p-1}}.
\]
The matrix identity
\[
        I-M_p^{2m}
        =
        (I-M_p^2)Q_m(M_p^2)
\]
gives
\[
        \det(I-M_p^{2m})
        =
        \det(I-M_p^2)\det Q_m(M_p^2).
\]
Applying the Burau--Alexander formula to \(A_p^{2m}\) and \(A_p^2\),
with the same denominator \(1+t+\cdots+t^{p-1}\), yields
\[
        \Delta_{\Th(p,2m)}(t)
        \doteq
        \Delta_{\Th(p,2)}(t)\det Q_m(M_p(t)^2).
\]
By Proposition~\ref{prop:quotient-square},
\[
        \det Q_m(M_p(t)^2)=H_{p,m}(t)^2
\]
for some \(H_{p,m}(t)\in\ZZ[t,t^{-1}]\), proving the theorem.
\end{proof}

\begin{remark}\label{rem:det-shadow}
Specializing Theorem~\ref{thm:square-quotient} at \(t=-1\) explains the computational pattern
\[
        \det(\Th(p,2m))=\det(\Th(p,2))\cdot s_{p,m}^2,
\]
where \(s_{p,m}=|H_{p,m}(-1)|\).
Since \(\det(\Th(p,2))\) is the \(p\)-th Pell number, this is the determinant shadow of the stronger Alexander-polynomial factorization.
\end{remark}

\section{The base case \texorpdfstring{\(q=2\)}{q=2} modulo 2}\label{sec:q2-mod2}

We now prove that \(\Delta_{\Th(p,2)}(t)\) is not a Laurent square.

The survey of Di Prisa--\c{S}avk records, following Kr\"otenheerdt and Takemura, a Seifert matrix \(V_{p,q}\) for the canonical fiber surface of \(\Th(p,q)\) \cite{Kroetenheerdt,Takemura} (see also \cite[Section~4.6]{DPSsurvey}).  The matrix is written in blocks of size \((q-1)\times(q-1)\), where \(A_q\) is the Seifert matrix of the torus link \(T(2,q)\).  For \(p\) odd, the diagonal blocks alternate between \(A_q\) and \(-A_q^T\), while the first subdiagonal blocks are \(-A_q\), and all other blocks are zero.

For \(q=2\), the block size is \(1\) and
\[
        A_2=(-1).
\]
Consequently, modulo \(2\), all signs disappear and
\[
        V_{p,2}\equiv L_{p-1}\pmod 2,
\]
where \(L_n\) denotes the \(n\times n\) lower bidiagonal matrix
\[
        L_n=
        \begin{pmatrix}
        1 & 0 & 0 & \cdots & 0\\
        1 & 1 & 0 & \cdots & 0\\
        0 & 1 & 1 & \cdots & 0\\
        \vdots & \ddots & \ddots & \ddots & \vdots\\
        0 & \cdots & 0 & 1 & 1
        \end{pmatrix}.
\]

\begin{remark}\label{rem:seifert-convention}
Some references use \(\det(tV-V^T)\) and others use \(\det(V-tV^T)\) for the Alexander polynomial \cite{Lickorish1997}.  These differ by a Laurent unit after normalization.  In the modulo \(2\) argument below, a transpose convention replaces \(L_n+tL_n^T\) by \(tL_n+L_n^T\), which gives the same conclusion up to multiplication by a monomial.  Thus the square obstruction is independent of these conventions.
\end{remark}

\begin{proposition}\label{prop:q2-mod2}
Let \(p\geq 3\) be odd.  Then, up to multiplication by a Laurent monomial,
\[
        \Delta_{\Th(p,2)}(t)
        \equiv
        1+t+t^2+\cdots+t^{p-1}
        \pmod 2.
\]
\end{proposition}

\begin{proof}
Set \(n=p-1\). Using
\(\Delta_K(t)\doteq\det(V-tV^T)\) and reducing modulo \(2\), we obtain,
up to multiplication by a Laurent monomial,
\[
        \Delta_{\Th(p,2)}(t)
        \equiv
        \det(L_n+tL_n^T)
        \pmod 2.
\]
The matrix \(L_n+tL_n^T\) is the tridiagonal matrix
\[
        \begin{pmatrix}
        1+t & t & 0 & \cdots & 0\\
        1 & 1+t & t & \cdots & 0\\
        0 & 1 & 1+t & \ddots & 0\\
        \vdots & \ddots & \ddots & \ddots & t\\
        0 & \cdots & 0 & 1 & 1+t
        \end{pmatrix}
\]
over \(\mathbb F_2[t]\).  Let
\[
        D_n(t)=\det(L_n+tL_n^T).
\]
Then
\[
        D_0(t)=1,
        \qquad D_1(t)=1+t,
\]
and the tridiagonal determinant recurrence gives
\[
        D_n(t)=(1+t)D_{n-1}(t)+tD_{n-2}(t)
\]
over \(\mathbb F_2[t]\).

We claim that
\[
        D_n(t)=1+t+\cdots+t^n.
\]
Let \(S_n(t)=1+t+\cdots+t^n\).  The initial conditions agree.  Moreover, over \(\mathbb F_2[t]\),
\[
\begin{aligned}
        (1+t)S_{n-1}(t)+tS_{n-2}(t)
        &=S_{n-1}(t)+tS_{n-1}(t)+tS_{n-2}(t) \\
        &=1+t+\cdots+t^n
        =S_n(t),
\end{aligned}
\]
because the middle terms cancel in characteristic \(2\).  Thus \(D_n(t)=S_n(t)\), and the claim follows with \(n=p-1\).
\end{proof}

\begin{corollary}\label{cor:q2-not-square}
For every odd \(p\geq 3\), \(\Delta_{\Th(p,2)}(t)\) is not a Laurent square up to unit in \(\ZZ[t,t^{-1}]\).
\end{corollary}

\begin{proof}
Suppose, for contradiction, that
\[
        \Delta_{\Th(p,2)}(t)
        =
        \varepsilon t^aG(t)^2,
        \qquad
        \varepsilon\in\{\pm1\},
        \quad
        G(t)\in\ZZ[t,t^{-1}].
\]
By Proposition~\ref{prop:q2-mod2}, there is an integer \(b\) such that
\[
        \overline{\Delta}_{\Th(p,2)}(t)
        =
        t^b(1+t+\cdots+t^{p-1})
\]
in \(\mathbb F_2[t,t^{-1}]\). Reducing the assumed square
factorization modulo \(2\) therefore gives
\[
        t^{b-a}(1+t+\cdots+t^{p-1})
        =
        \overline G(t)^2.
\]
The exponents occurring on the left are consecutive and hence include
both parities, whereas every Laurent square over \(\mathbb F_2\), up
to multiplication by a monomial, has all its exponents in one parity
class by Lemma~\ref{lem:F2-square-parity}. This is a contradiction.
\end{proof}

\section{Failure of equivariant rational sliceness}\label{sec:even-failure}

\begin{proof}[Proof of Theorem~\ref{thm:even-obstruction}]
Let \(q=2m\).  By Theorem~\ref{thm:square-quotient},
\[
        \Delta_{\Th(p,2m)}(t)
        \doteq
        \Delta_{\Th(p,2)}(t)H_{p,m}(t)^2
\]
for some \(H_{p,m}(t)\in\ZZ[t,t^{-1}]\).  By
Corollary~\ref{cor:q2-not-square}, \(\Delta_{\Th(p,2)}(t)\) is not a Laurent
square.

Since Alexander polynomials of knots are nonzero, the factorization above
implies that \(H_{p,m}(t)\neq0\). Because
\(\ZZ[t,t^{-1}]\) is a unique factorization domain, choose an irreducible
factor whose exponent in \(\Delta_{\Th(p,2)}(t)\) is odd. Its exponent in
\[
        \Delta_{\Th(p,2)}(t)H_{p,m}(t)^2
\]
is the sum of that odd exponent and an even integer, and hence remains odd.
Therefore \(\Delta_{\Th(p,2m)}(t)\) is not a Laurent square.

Now suppose \(\Th(p,2m)\) were equivariantly \(\QQ\)-slice with respect to some
strong inversion.  Di Prisa--\c{S}avk's equivariant Fox--Milnor obstruction
would imply that \(\Delta_{\Th(p,2m)}(t)\) is a Laurent square
\cite{DPSequivariant}, contradicting the previous paragraph.  Thus
\(\Th(p,2m)\) is not equivariantly \(\QQ\)-slice with respect to any strong
inversion.
\end{proof}

\begin{proof}[Proof of Corollary~\ref{cor:even-not-klein}]
If \(\Th(p,q)\) were Klein amphichiral, then Di Prisa--\c{S}avk's theorem that
Klein amphichiral knots are equivariantly \(\QQ\)-slice would imply equivariant
\(\QQ\)-sliceness for the corresponding strong inversion \cite{DPSequivariant}.
This contradicts Theorem~\ref{thm:even-obstruction}.  Hence \(\Th(p,q)\) is not Klein
amphichiral.
\end{proof}

\begin{proof}[Proof of Theorem~\ref{thm:parity-dichotomy}]
If \(q\) is odd, then \(q\geq3\), and the assertion follows from
Theorem~\ref{thm:odd-klein} and Corollary~\ref{cor:equivariant-Q-slice}. If
\(q\) is even, then the assertion follows from Theorem~\ref{thm:even-obstruction}
and Corollary~\ref{cor:even-not-klein}.
\end{proof}

\section{Examples and computational checks}\label{sec:examples}

The theorem explains the determinant data obtained experimentally.  For odd \(p\) and even \(q=2m\), Theorem~\ref{thm:square-quotient} implies
\[
        \det(\Th(p,2m))
        =
        \det(\Th(p,2))\,s_{p,m}^2.
\]
Dowdall--Mattman--Meek--Solis computed
\[
        \det(\Th(p,2))=P_p,
\]
where \(P_p\) is the \(p\)-th Pell number; this formula is also recorded in the Di Prisa--\c{S}avk survey \cite{DPSsurvey,DMMS}.  Thus the squarefree part of \(\det(\Th(p,2m))\) is the squarefree part of \(P_p\).

For example, computations give
\[
\begin{array}{c|c|c|c}
(p,q) & \det(\Th(p,q)) & \sqfree\!\bigl(\det(\Th(p,q))\bigr) & \det(\Th(p,q))/P_p \\
\hline
(5,4)  & 3509              & 29   & 11^2 \\
(7,4)  & 284089            & 1    & 41^2 \\
(9,4)  & 23057865          & 985  & 153^2 \\
(11,4) & 1871801381        & 5741 & 571^2
\end{array}
\]
The case \(p=7\) is instructive: the determinant is a square, but the Alexander polynomial is not.  Indeed,
\[
        \Delta_{\Th(7,2)}(t)
        \doteq
        (t^3-6t^2+5t-1)(t^3-5t^2+6t-1),
\]
which reduces modulo \(2\) to
\[
        1+t+t^2+t^3+t^4+t^5+t^6.
\]
Thus the full Alexander-polynomial obstruction detects exactly what the determinant misses.

\section{Conclusion}\label{sec:conclusion}

We have proved a parity dichotomy for odd-stranded Turk's head knots. If
\(p\geq3\) is odd, \(q\geq2\), and \(\gcd(p,q)=1\), then the behavior of
\(\Th(p,q)\) with respect to Klein amphichirality and equivariant rational
sliceness is controlled exactly by the parity of \(q\).

For odd \(q\), we constructed an explicit compatible \(V_4\)-symmetry on the
signed, arrowed Gauss diagram, realized it through the associated abstract link
diagram on the supporting sphere, and thickened it to a commuting pair of
ambient involutions of \(S^3\). This proves that \(\Th(p,q)\) is Klein
amphichiral, and hence equivariantly \(\mathbb Q\)-slice with respect to the
constructed strong inversion.

For even \(q\), we proved that \(\Th(p,q)\) fails the Alexander-polynomial
square condition required by equivariant \(\mathbb Q\)-sliceness. The key
algebraic input was the factorization
\[
        \Delta_{\Th(p,2m)}(t)
        \doteq
        \Delta_{\Th(p,2)}(t)H_{p,m}(t)^2,
\]
together with a modulo \(2\) proof that \(\Delta_{\Th(p,2)}(t)\) is not a
Laurent square. By the equivariant Fox--Milnor obstruction of
Di Prisa--\c{S}avk, \(\Th(p,2m)\) is therefore not equivariantly
\(\mathbb Q\)-slice with respect to any strong inversion. Since every Klein
amphichiral knot is equivariantly \(\mathbb Q\)-slice, the even-parameter knots
are not Klein amphichiral.

Thus, for odd \(p\), the coprime Turk's head knots satisfy
\[
        \Th(p,q)\text{ is Klein amphichiral}
        \quad\Longleftrightarrow\quad
        q\text{ is odd},
\]
and, for this family, they admit an equivariantly \(\mathbb Q\)-slice strong
inversion exactly in the odd-parameter case.

\section*{Acknowledgments}
The author is deeply grateful to Professor Rama Mishra for her generous
mentorship, for many valuable discussions, and for introducing the
author to the study of knot theory.  The author also thanks Professor
Seiichi Kamada for helpful comments on terminology and exposition,
especially concerning the use of standard disk--band surface language,
and Professor Makoto Sakuma for insightful feedback on the symmetry
schematic and for clarifying the fixed-point geometry of the strong
inversion.

\bibliographystyle{alpha}
\bibliography{references}

\end{document}